\documentclass[12pt]{article}
\usepackage{amsmath,amssymb,graphicx,psfrag}
\newenvironment{proof}{\noindent{\bf Proof:} \hspace*{1em}}{
    \hspace*{\fill} $\square$\medskip }
\newcommand{\old}[1]{}

\newtheorem{theorem}{Theorem}[section]

\newtheorem{prop}{Proposition}

\newcommand{\eps}{\epsilon}
\newcommand{\Z}{\mathbb Z}

\newcommand{\E}{\mathbb E}
\newcommand{\Prob}{\mathbb P}

\begin{document}
\title{Monotone loop models and rational resonance}
\author{Alan Hammond\thanks{Mathematics Department,
NYU, New York.} \and Richard Kenyon\thanks{Mathematics Department,
Brown University, Providence, RI.} }
 \maketitle

\begin{abstract}
Let $T_{n,m}=\mathbb Z_n\times\mathbb Z_m$, and define a random mapping
$\phi\colon T_{n,m}\to T_{n,m}$ by $\phi(x,y)=(x+1,y)$ or $(x,y+1)$
independently over $x$ and $y$ and with equal probability.
We study the orbit structure of such ``quenched random walks'' $\phi$
in the limit $m,n\to\infty$,
and show how it depends sensitively on the ratio $m/n$.
For $m/n$ near a rational $p/q$, we show that 
there are likely to be on the order
of $\sqrt{n}$ cycles, each of length $O(n)$, whereas for 
$m/n$ far from any rational with small 
denominator, there are a bounded number of cycles,
and for typical $m/n$ each cycle has length on the order of $n^{4/3}$.
 \end{abstract}
 
\setlength{\baselineskip}{16pt}

\section{Introduction}
We study a model of monotone non-intersecting lattice paths in $\mathbb{Z}^2$. While this is a classically studied model in statistical mechanics, related to Dyson Brownian motion and random matrices,
there are few studies concerned with the influence of the boundary
conditions at the critical point of the model. The authors of \cite{KW}
studied a similar model in this setting, finding some surprising ``resonance''
phenomena, which showed how the shape of the domain---in particular
the rationality of the aspect ratio---had a strong
influence on the partition function and other observables in the model.
Here we study a model very closely related to that in \cite{KW}, 
on which we can obtain
more accurate and complete information using simple probabilistic methods.
One of our goals is to explain some of the conjectured behavior in 
\cite{KW}. However we feel that our model is of primary interest
as an example of a quenched random dynamical system for which
a fairly complete analysis can be obtained.

For positive integers $m,n$ 
let $\Gamma_{n,m}$ be the sublattice of $\Z^2$ generated
by $(n,0)$ and $(0,m)$. 
Let $T_{n,m}=\Z^2/\Gamma_{n,m}\cong \Z/n\Z\times \Z/m\Z$, the $n,m$-torus.
We will consider configurations consisting of collections of vertex-disjoint monotone lattice cycles on $T_{n,m}$. The law on such configurations that we consider has two simple definitions.

Firstly, we define $\phi: T_{n,m} \to T_{n,m}$ by setting $\phi(x,y)$ to be equal, with equal probability, either to $(x+1,y)$ or $(x,y+1)$, these choices being made independently over the $nm$ vertices of $T_{n,m}$. We call $\phi$ a {\bf quenched random walk}.
Each $\phi$ represents a dynamical system, with at least one periodic orbit. The law on disjoint unions of cycles that we consider is given by the collection of distinct orbits of this randomly selected $\phi$.  

The dynamical system $\phi$ is equivalent to another model, the {\bf cycle-rooted spanning forest}. We may make $T_{n,m}$ into a directed graph, with each vertex $(x,y)$ having two outgoing edges, that point to $(x+1,y)$ and $(x,y+1)$. A cycle-rooted spanning forest is a subgraph in which each vertex has one outgoing edge. 
The components in these subgraphs may be several, but each component contains a single directed cycle, which is topologically nontrivial. The remaining edges of the component form in-directed trees, attached to this cycle.
The uniform probability measure on cycle-rooted spanning forests is called the CRSF measure, $\mu_{CRSF}$. Each component is referred to as a cycle-rooted spanning tree. See Figure \ref{cycles}.

It is not hard to see that the cycles of the quenched random 
walk $\phi$ are precisely
the cycles in the CRSF model. We simply define $\phi(x,y)$ to be the 
vertex pointed to by the random outgoing edge from $(x,y)$ in the CRSF.

The model is closely related to the monotone non-intersecting lattice path (MNLP) model, which was studied in \cite{KW}. The state space of the measure again consists of collections of disjoint, monotone northeast-going lattice cycles. A given configuration is chosen according to a Boltzmann measure $\mu$ at temperature $T$, which is the probability measure assigning to a configuration a probability proportional to $e^{- \frac{E_b N_b + E_c N_c}{T}}$. Here, 
$E_b$ and $E_c$ are positive constants, and $N_b$ and $N_c$ denote the total number of eastgoing or northgoing steps in the configuration. 

In \cite{KW}, the behaviour of the MNLP model was examined near its critical temperature $T$, which is the temperature 
at which $e^{-E_b/T} + e^{-E_c/T} = 1$. In the critical case, a configuration is being chosen with probability proportional to $b^{N_c}c^{N_c}$, where $b,c$ satisfy $b + c = 1$. 

The two models, CRSF and MNLP, have similarities in their definition, and we will shortly explain their connection more precisely. As \cite{KW} showed, MNLP is amenable to an exact solution analysis using Kasteleyn theory. The CRSF model, on the other hand, has a dynamical definition that permits a more geometric discussion of its behaviour.

The authors of \cite{KW} determine that the behaviour of MNLP depends sensitively, for large tori, on the aspect ratio $m/n$
of the torus, with radical changes in behaviour occurring near rational values for this ratio. In Theorem 1 of \cite{KW}, for example, an asymptotic expression for the value of partition function (whose definition we will shortly give), and for the mean and variance of the number of edges present in an MNLP configuration is computed. Figure 3 of \cite{KW} illustrates a conjecture of \cite{KW}: the number of edges typically present in an MNLP configuration appears to be highly peaked when the aspect ratio is precisely equal to a given small rational, it experiences a rapid decay if the aspect ratio is slightly increased or decreased, and, strikingly, it increases again if the change in aspect ratio is further accentuated, while still remaining ``far'' from other rationals.

Consider the critical case of MNLP where $b = c  = 1/2$ (we might consider other values of $b,c$ such that $b+c=1$, but this is essentially indistinguishable, for asymptotic behaviour, from changing the aspect ratio of the torus). In this case, MNLP assigns a weight of $2^{-\vert \mathcal{C} \vert}$ to any configuration $\mathcal{C}$, a disjoint union of north- and east-going cycles,
where $|\mathcal{C}|$ is the total number of edges of $\mathcal{C}$.  
We define $Z_{MNLP}=\sum_{\mathcal{C}}2^{-|\mathcal{C}|}$, where the sum is over all 
configurations, so that
the probability of a configuration in $2^{-|\mathcal{C}|}/Z_{MNLP}$. The quantity
$Z_{MNLP}$ is the {\bf partition function} of the MNLP model.

We now discuss the relation between the two models, MNLP and CRSF.
Consider a variant of CRSF, called oriented CRSF, which is given by the uniform measure on cycle-rooted spanning forests, each of whose cycles is given an orientation (either to the northeast, or the southwest). We can exhibit a measure-preserving correspondence between MNLP and oriented CRSF. Indeed, if we take a sample of MNLP, and orient each of its cycles in the southwesterly direction, and then assign to each remaining vertex in the torus an edge pointing to the east or to the north, indepedently and with equal probability, declaring that the newly formed cycles, if any, are to be oriented in the northeasterly direction, we claim that the resulting law is the uniform measure on oriented cycle-rooted spanning forests. We see this as follows: for a given cycle-rooted spanning forest whose cycles are oriented, the procedure will alight on this configuration if, firstly, the sample of MNLP happens to pick its set $\mathcal{C}$ of southwesterly oriented cycles and, secondly, the correct choice of north or east outgoing edges is made in each of the remaining vertices. The probability of the first event is $2^{-\vert \mathcal{C} \vert}/Z_{MNLP}$, while the second then occurs with probability $2^{-(mn - \vert \mathcal{C} \vert)}$. So the probability of obtaining the given oriented cycle-rooted spanning forest, which is $2^{-mn} Z_{MNLP}^{-1}$, 
does indeed not depend on the choice of oriented CRSF. 

This argument also demonstrates that the number of oriented CRSFs is given by $2^{mn} Z_{MNLP}$. We also know that it is equal to the sum over CRSFs of $2^{\textrm{cycles}}$, or to $2^{mn} \mathbb{E}_{CRSF}(2^{\textrm{cycles}})$, since there a total of $2^{mn}$ CRSFs.    

We summarise these deductions: 
\begin{theorem}\label{MNLPthm}
The MNLP model is in measure-preserving correspondence 
with oriented CRSF, and
\begin{equation}\label{Zequality}
\E_{CRSF}(2^\mathrm{cycles})= Z_{MNLP},
\end{equation}
That is, the partition function for the MNLP model
is the expected value of $2$ to the number of components 
in the CRSF model.
\end{theorem}
In \cite{KW} it was conjectured that $Z_{MNLP}\geq2$.
This follows trivially from \eqref{Zequality}.

\begin{figure}[htbp]
\centerline{\includegraphics[height=0.51\textwidth]{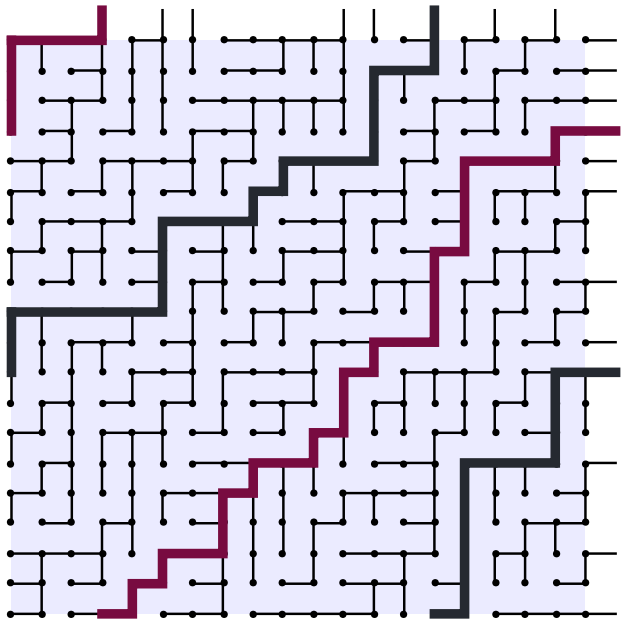}}
\caption{A quenched random walk and its orbit structure, with the
cycles highlighted.
\label{cycles}}
\end{figure}

In this paper we study the number and the length of cycles in CRSF, as well as their homology class. Our results, which we shortly outline, yield a geometric understanding of CRSF: the sharp changes in behaviour that occur near rational choices for the aspect ratio for a large torus, and the remaining generic case in which the aspect ratio is not close to any small rational.
As well as treating this simple and natural model, our results also have implications for MNLP.
A reweighting as in (\ref{Zequality}) is required to make inferences about MNLP, because our results treat the unoriented version of CRSF. It is easily seen, however, that the number of cycles present in oriented and unoriented CRSF differs by a constant factor, uniformly in $n$ and $m$. This means that, 
for example, 
we have found a qualitative explanation 
for the secondary spikes present in the mean total length of cycles in MNLP near a given rational aspect ratio that is described in Section 6 of \cite{KW}.

\begin{figure}[htbp]
\centerline{\includegraphics[height=0.41\textwidth]{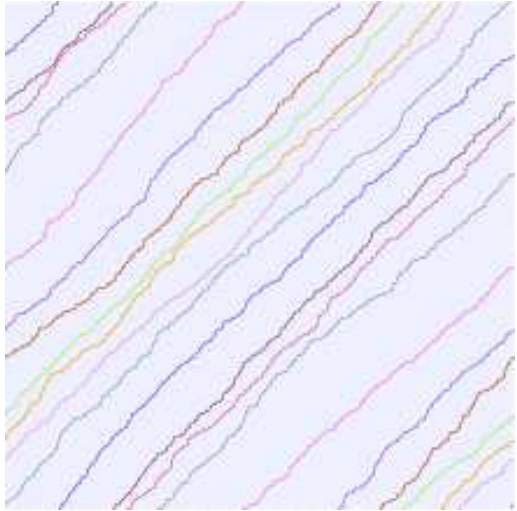}}
\caption{Cycles in a sample when $n=m=1000$.
\label{cycles10001000}}
\end{figure}

\begin{figure}[htbp]
\centerline{\includegraphics[height=0.31\textwidth]{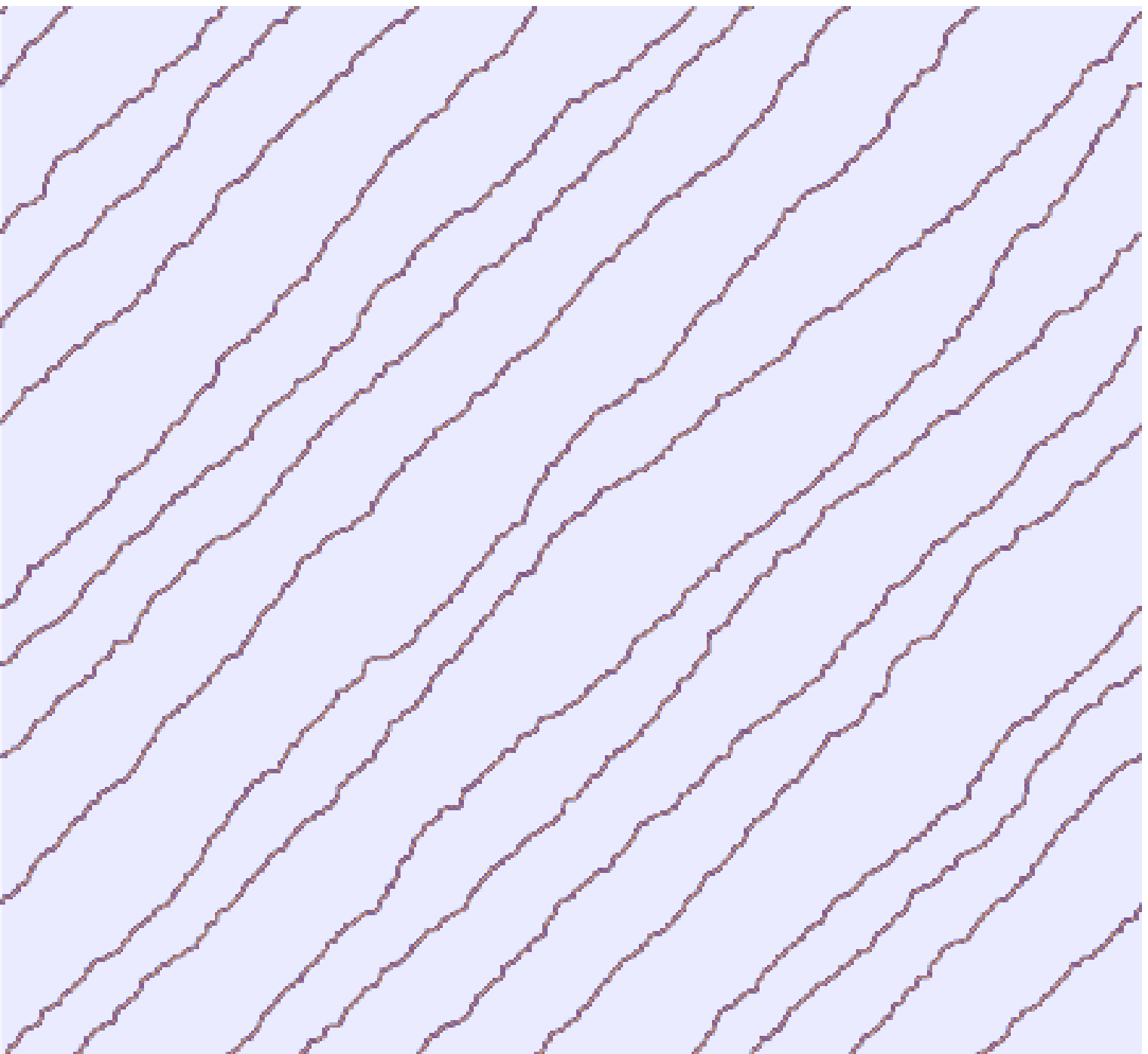}\hspace{0.5cm}\includegraphics[height=0.31\textwidth]{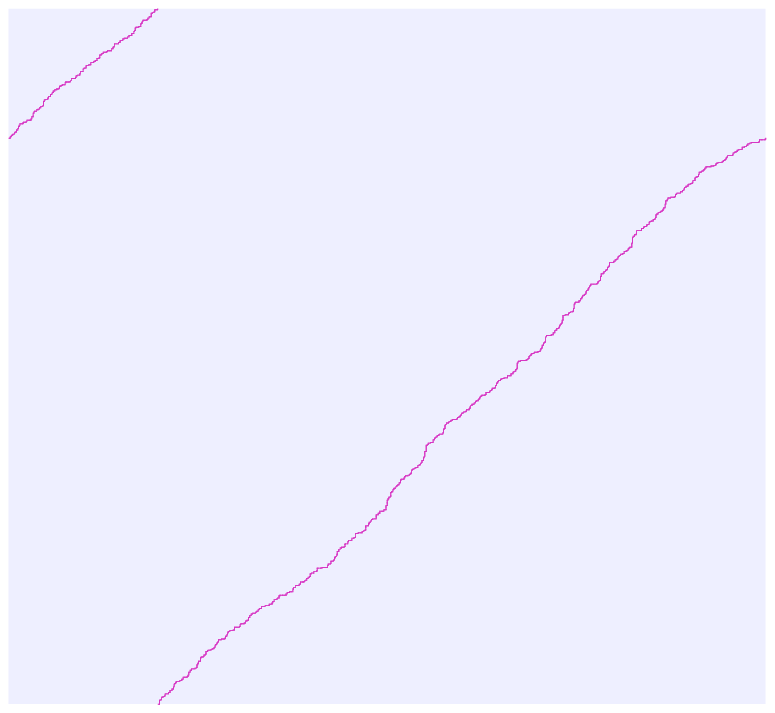}}
\caption{Cycles in two samples when $n=1090, m=1000$. In the first there is one
cycle with homology class $(9,10)$.
\label{cycles10901000}}
\end{figure}

\subsection{Outline}

In section \ref{sqrt}, we prove two propositions, 
the first detailing the behaviour of the number of cycles 
when $m - n = O(n^{1/2})$, and the second when
$n^{1/2} << \big\vert m - n \big\vert << \sqrt{2n \log n}$. 
We show that, in the first case, $\Theta(n^{1/2})$ 
cycles are likely to be present in the CRSF, while this 
number experiences a rapid decay as we enter the second case. 
In section \ref{sqrtlog}, we turn to the behaviour of the model 
when $m = n + C \sqrt{n\log n}$, where $C > \sqrt{2}$ is a fixed constant, 
showing that precisely one loop is likely to exist, and that 
this loop is global in nature, having length $n^{3/2+o(1)}$. 
In Sections \ref{p/q}, we extend these results to the case $m/n$ close to a
rational $p/q$.  In Section \ref{irrationalsection},
we deal with the case that $m/n$ is not close to a rational
with small denominator. In this case, typically, a constant number of
cycles form, each having a length of order $n^{4/3}$.
These cycles cross the torus about $n^{1/3}$ times and divide it into pieces whose widths are about $n^{2/3}$.

\subsection{Acknowledgments}

We thank Yuval Peres for comments and references.
The work of the second author was supported by NSERC
and NSF. This work was started while both authors were at the University
of British Columbia.

\subsection{Notation}

We identify $T_{n,m}$ with the rectangle $\{0,\dots,n-1\}\times\{0,\dots,m-1\}$
in $\Z^2$.  A closed orbit (or \emph{cycle}) has {\bf homology class $(p,q)$} if it crosses
any horizontal line $q$ times and any vertical line $p$ times.
We refer to such an orbit as a {\bf $(p,q)$-cycle}.
Two disjoint closed orbits necessarily have the same homology class,
and $p,q$ are necessarily relatively prime, since the orbits are simple closed curves. 
The length of a $(p,q)$-cycle is $np + mq$ edges; it necessarily has $np$ horizontal
edges and $mq$ vertical edges.

A {\bf strand} of a cycle $C = \big\{ c_0,c_1,\ldots,c_n \big\}$, with $c_n = c_0$, is 
defined to be a subpath $\big\{ c_i,c_{i+1},\ldots,c_{j-1},c_j\big\}$ between two consecutive passes of the line $x=\frac12$,
that is, such that $c_{i-1}$ has zero $x$-coordinate, $c_i$ 
has $x$-coordinate $1$, and similarly for $c_{j},c_{j+1}$. A cycle is partitioned by the set of its strands.  

For a particular realization of $\phi$,
and for $(p,q) \in \mathbb{N}^2$,
we write $N_{(p,q)}$ for the number of cycles of homology class $(p,q)$,
and  $N= \sum_{(p,q) \in \mathbb{N}^2} N_{(p,q)}$ for the total number
(note that exactly one term in the sum is nonzero). 

For each $(i,j) \in \{0,\ldots,n-1\} \times \{0,\ldots,m-1\}$, associate a random walk
$W_{i,j}: \mathbb{N} \to T_{n,m}$, starting at $(i,j)$ and
whose steps are independently up and to the right with equal probability. 
We call such a random walk an
{\bf up-right random walk}.

The CRSF can be obtained by iteratively running 
the random walks $\big\{ W_{ij} \big\}_{0\leq i \leq n-1,0 \leq j \leq m-1}$, stopping each 
when it intersects its trace or the trace of those run earlier. 
We will form the configuration in this way, or by some variation of this approach.
Although in constructing the CRSF, we have no cause to 
examine the walks after they intersect themselves, the 
independence properties of the walks without stopping 
will be useful in the proofs of the propositions.

Similarly to the definition of a strand of a cycle, we say that a walk $W_{x,y}$ performs a traversal on the interval $\big\{ t_1,\ldots t_2 - 1 \big\}$ if the $x$-coordinate of $W_{x,y}$ is $m-1,0,m-1,0$
at times $t_1-1,t_1,t_2-1,t_2$ respectively (the first condition we omit if $t_1 = 0$), and $t_2$ is the first return to the line $x=0$ after $t_1$.
That is, a traversal is a horizontal crossing of the torus by the walk.

\section{The primary spike and the valley when $m\approx n$}\label{sqrt}

We discuss in this section the case when $m-n=O(\sqrt{n})$.
In this case there are many $(1,1)$-cycles. This case is generalized to 
$m/n\approx p/q$ in Section \ref{p/q}.

\begin{prop}\label{propone}
Let $\rho\in(0,\infty)$ be fixed and $m = n + \rho \sqrt{n}(1+o(1))$.
There exists $c = c(\rho) > 0$ with the following property. 
The probability that each closed orbit has homology class $(1,1)$ exceeds $1 - \exp \big\{ - c n^{1/2} \big\}$. The number $N_{(1,1)}$ of such orbits
satisfies
$$
\Prob\Big(N_{(1,1)} > cn^{1/2}\Big) \geq 1 - \exp \Big\{ - c n^{1/2} \Big\}
$$
for all sufficiently large $n$.
Moreover there exists a (large) $C>0$ so that we have
$$
\Prob\Big(N_{(1,1)} > C n^{1/2}\Big)  \leq \exp \Big\{ - C^{-1} n^{1/2} \Big\}   
$$
for all sufficiently large $n$.
\end{prop}

\begin{proof} 
Let $K>0$ be large but fixed as $n\to\infty$. We partition the torus into strips of width $2K\sqrt{n}$, parallel
to the closed path $y=(m/n)x$.
That is, let $A_i, i \in \big\{ 0,\ldots, \lfloor \frac{m}{2K\sqrt{n}} \rfloor - 1 \big\}$
denote the set of vertices in the $i$th strip:
\begin{displaymath}
  A_i = \Big\{ (x,y) \in T_{n,m}: 2K n^{1/2} i \leq \vert y - (m/n)x \vert 
 < 2K n^{1/2} (i+1)  \Big\}.
\end{displaymath} 
The last strip, $A_{\lfloor \frac{m}{2K\sqrt{n}} \rfloor}$, may be thinner than the others, but this does not matter for our purpose. 

We will form the CRSF configuration in the following way. Let $z_i = \lfloor K n^{1/2} (2i+1) \rfloor$ be the $y$-coordinate of the point on the line $x=0$ in the 
center of the $i$th strip. 
We will run the walks $W_{0,z_i}$, $i=1,2,\dots$ in increasing order, 
stopping any such walk at the stopping time $\sigma_i$, where
$$
 \sigma_i = \min \Big\{ j \geq 0 : W_{0,z_i}(j) \not\in A_i, \textrm{ or   $W_{0,z_i}(j) = W_{0,z_i}(t)$ for some $t < j$} \Big\}
$$
denotes the first time at which $W_{0,z_i}$ either leaves $\partial A_i$ or hits its own trace. 

After these segments of random walks have been run, we choose an arbitrary order of successive sites as the initial locations of further random walks, until a cycle-rooted 
spanning forest has been determined.

Let  $E_i, i \in \big\{ 0,\ldots,\lfloor \frac{m}{2K\sqrt{n}} \rfloor \big\}$
denote the event that  the walk $W_{0,z_i}$ remains in $A_i$ during its first two traversals, and that its first return to the line $\{ x = 0 \}$ occurs at a $y$-coordinate exceeding $z_i$, while its second return 
has a $y$-coordinate at most $z_i$. We claim that
$\mathbb{P}(E_i) > c = c(\rho) > 0$.
Indeed, let $\big\{ X_j: j \in \mathbb{N} \big\}$ denote the number of upward displacements made by $W_{0,z_i}$ between its $j-1$-st and $j$-th rightward displacements. The event that $W_{0,z_i}$ remains in $A_i$ during its first two traversals occurs precisely when
\begin{equation}\label{xicond}
 \Big\vert  \sum_{k=1}^j X_k \, - \frac{mj}{n} \Big\vert
 \leq K n^{1/2}
\end{equation}
for each $j \in \big\{ 1,\ldots, 2n \big\}$. The condition on $m$ implies that
$\big\vert \frac{mj}{n} - j \big\vert \leq 2 \sqrt{n} \rho \big( 1 + o(1) \big)$ for such values of $j$,  
from which we learn that (\ref{xicond}) is implied by the inequalities 
\begin{equation}\label{neweqn}
 \Big\vert  \sum_{k=1}^j X_k \, -  j \Big\vert
 \leq \big(K - 3 \rho \big) n^{1/2}
\end{equation}
each being satisfied, for $j \in \{ 1,\ldots, 2n \}$.
For the occurrence of $E_i$, we require in addition that
\begin{equation}\label{xknm}
 \sum_{k=1}^n X_k \, - \, m \in \Big[ 0, K n^{1/2} \Big]
\end{equation}
and that
\begin{equation}\label{xktnm}
 \sum_{k=1}^{2n} X_k \, - \, 2m \in \Big[ - K n^{1/2}, 0  \Big].
\end{equation}

 Equivalent to (\ref{neweqn}) is the assertion that the random walk $\sum_{i=1}^j {2^{-1/2}(X_i - 1)}$, that has a step distribution with mean zero and variance one, has a maximum absolute value of at most $(K - 3 \rho)(n/2)^{1/2}$, among $j \in \{ 1,\ldots, 2n \}$. As (\ref{neweqn}) is sufficient for (\ref{xicond}), so are (\ref{xknm})
and (\ref{xktnm}) implied by
\begin{equation}\label{sqone}
 \sum_{k=1}^n 2^{-1/2} \big( X_k - 1 \big) \in \Big[  \frac{3 \rho}{\sqrt{2}} n^{1/2}, \frac{K - 3\rho}{\sqrt{2}} n^{1/2} \Big]
\end{equation} 
and
\begin{equation}\label{sqtwo}
 \sum_{k=1}^{2n} 2^{-1/2} \big( X_k - 1 \big) \in \Big[ - \frac{K - 3\rho}{\sqrt{2}} n^{1/2}  , - \frac{3 \rho}{\sqrt{2}} n^{1/2} \Big].
\end{equation} 
We have then that $\mathbb{P}(E_i) \geq p_n$, where $p_n$ is the probability that each of the conditions (\ref{neweqn}), (\ref{sqone}) and (\ref{sqtwo}) is satisfied.
By Donsker's theorem (\cite{Durrett}, page 365), we have that $p_n \to p$ as $n \to \infty$, where 
\begin{equation}\label{bbeqn}
p = \mathbb{P}\Big( \Big\{ \sup_{t \in [0,2]}\big\vert B(t) \big\vert \leq \frac{K - 3 \rho}{\sqrt{2}} \Big\}  \cap \Big\{ B(1) \in \Big[  \frac{3 \rho}{\sqrt{2}} , \frac{K - 3 \rho}{\sqrt{2}} \Big] \Big\} \cap \Big\{  B(2) \in \Big[  -  \frac{K - 3 \rho}{\sqrt{2}} , \frac{- 3 \rho}{\sqrt{2}} \Big]  \Big\} \Big),
\end{equation}
with $B:[0,\infty) \to \mathbb{R}$ denoting a standard Brownian motion. 
Note that $p > 0$: 
the event
$$
 \Big\{ B(1) \in \Big[  \frac{3 \rho}{\sqrt{2}} , \frac{K - 3 \rho}{\sqrt{2}} \Big] \Big\} \cap \Big\{  B(2) \in \Big[  -  \frac{K - 3 \rho}{\sqrt{2}} , \frac{- 3 \rho}{\sqrt{2}} \Big]  \Big\}
$$
occurs with positive probability, because $B(1)$
and $B(2) - B(1)$
are independent normal random variables.
Conditionally on the pair 
$\big( B(1),B(2) \big)$ taking a given value in the set 
$\big[ \frac{3 \rho}{\sqrt{2}}  , \frac{K - 3 \rho}{\sqrt{2}}   \big] \times 
 \big[   -  \frac{K - 3 \rho}{\sqrt{2}} , \frac{- 3 \rho}{\sqrt{2}}  \big]$,
there is a uniformly positive probability that the first condition in the event on the right-hand-side of (\ref{bbeqn}) is satisfied, as we see from  
 the law of the maximum of a Brownian bridge (\cite{Durrett}, exercise 8.2, page 391).

Thus, indeed, each event $E_i$ has a positive probability $>c(\rho)$, bounded below independently of $n$ and $i$.

If $E_i$ occurs, then, at the stopping time $\sigma_i$, the walk $W_{0,z_i}$ hits its own trace. The choice of the order of the walks in the formation of the CRSF configuration ensures that this event produces a $(1,1)$-cycle in $A_i$. 
Moreover, the events $E_i$ are pairwise independent. 
Thus, the number $N_{(1,1)}$ of $(1,1)$-cycles is 
bounded below by a binomial random variable with parameters $\lfloor \frac{m}{2K\sqrt{n}} \rfloor$ and $c(\rho)$. If the configuration contains one $(1,1)$-cycle, then all the other cycles are also of this type, so the absence of a $(1,1)$-cycle implies that $E_i$ does not occur for any $i \in \big\{ 0,\ldots, \lfloor m/(2Kn^{1/2}) \rfloor - 1 \big\}$.
We infer the first two statements of the proposition.

To treat the last assertion, we form the CRSF by running the random walk $W_{00}$ and adding an edge $(x,y)(x',y')$ traversed by $W_{00}$ to the configuration on each occasion for which the site $(x,y)$ is visited by $W_{00}$ for the first time, 
until a cycle-rooted spanning forest is formed. We may assume
that at least one $(1,1)$-cycle is formed, or, equivalently, that every cycle in the configuration is a $(1,1)$-cycle.

Define the wraparound time of the walk $W_{00}$ to be the earliest time $t$ such that the set $W_{00}[0,t]$ of vertices visited by the walk up to time $t$ has the property that every $(1,1)$-cycle in $T_{n,m}$ intersects $W_{00}[0,t]$.

We record the successive maxima and minima of the $y$-coordinate of the intersection of the walk $W_{00}$ with the line $\big\{x = \frac12\big\}$ (i.e. the first horizontal
step after each visit to the line $x=0$).
Let $u_0 = X_1$ be the $y$-coordinate of the walk on the first occasion 
that it crosses the line $x=\frac12$. 
When the walk next crosses the line $\big\{ x=\frac12 \big\}$, its $y$-value, which, with the natural choice of shift by a multiple of $m$, we take to be $\sum_{i=1}^{n+1}{X_i} \, - m$, may or may not exceed $u_0$.
If it is greater than $u_0$, we record its value as $u_1$, and, if it is smaller than $u_0$, we record it as $v_1$. We do not set the value of either $u_1$ or $v_1$ if   $\sum_{i=1}^{n+1}{X_i} \, - m$ equals $u_0$. 
The $y$-value of the walk on the occasion of the $k$-th return to the line $\big\{ x = 1/2 \big\}$ is given by  $\sum_{i=1}^{kn+1}{X_i} \, - km$.
We iteratively record the successive maxima of these statistics as $u_2,u_3,\ldots$ and the successive minima as $v_2,v_3,\ldots$. 

We no longer  record either maxima or minima on a return to $\big\{ x = 1/2 \}$ if this return occurs after the wraparound time.
Let $\{ u_0,\ldots, u_{J_1} \}$ and $\{ v_1,\ldots,v_{J_2} \}$ denote the final record. Let $Q$ denote the set of horizontal edges crossing $\{ x = 1/2 \}$
that are traversed by the walk at one of the recorded times.

By assumption, in each tree $T$ of the CRSF configuration lies a unique 
$(1,1)$-cycle,
and in this cycle
lies a unique horizontal edge 
$e = e(T)$ that crosses the line $\{ x = 1/2 \}$. 
Let $E$ denote the set of edges of the form $e(T)$ for some tree $T$ in the configuration. Let $e_0 \in E$ be the element in $E$ lying in the cycle in the configuration which is the last to be formed.

We claim that
\begin{equation}\label{ceeq}
E \setminus \big\{ e_0 \big\} \subseteq Q.
\end{equation}
To see this, note that if a $(1,1)$-cycle $C$ lies in the configuration, there exists a vertex $c \in C$ and $t,s \in \mathbb{N}$, $t < s$,
such that the walk $W_{00}$ makes its first and second visits to $c$ at times $t$ and $s$, the set  $W_{00}\big[ t+1,s \big]$ is the vertex set of $C$, and
\begin{equation}\label{wwdist}
  W_{00}\big[ t,s \big] \cap W_{00} \big[ 0,t-1 \big] = \emptyset.
\end{equation} 
 We call $t$ the start-time of the cycle $C$, and s, the end-time.

Note that the start-time of $C$ necessarily occurs before the wraparound time. 
The intervals of time during which distinct cycles of the CRSF configuration form being disjoint, we learn that every cycle except possibly that which forms last has an end-time that occurs before the wraparound time.

We claim that the cycle $C$, whose vertex set is $W_{00}[t+1,s]$,
either lies above or below the configuration $W_{00}[0,t-1]$
present prior to its formation. More precisely, the $y$-coordinate of 
every vertex $W_{00}(n)$,
$n \in \{t,\ldots,s \}$, exceeds the maximum $y$-coordinate of any
$W_{00}(m)$, $m \in \{0,\ldots,t-1 \}$, sharing its $x$-coordinate,
or the $y$-coordinate of 
every such vertex is less than the corresponding minimum. 
Indeed, this statement is readily verified from (\ref{wwdist}),
with the aid, for example, of the intermediate value theorem.

We have shown that every edge $e \in E$, except possibly $e_0$,
is traversed before the wraparound time, at a time which is recorded on either the $\{ u_i \}$ or $\{ v_i \}$ list. That is, we have obtained (\ref{ceeq}).

Note then that
\begin{equation}\label{fprw}
N  = \vert E \vert \leq  \vert Q \vert + 1 =  J_1+J_2 + 2,
\end{equation}
the inequality by (\ref{ceeq}).

We now show that, if $\rho = 0$, then for $i \in \{ 1,2 \}$, and if $\rho > 0$, for $i=1$,
\begin{equation}\label{riest}
 \mathbb{P} \big(  J_i > C n^{1/2} \big) \leq 
 \exp \Big\{ - C^{-1} n^{1/2} \Big\},
\end{equation}
for $C$ sufficiently large. Indeed, the sequence of increments $\{ u_{i+1} - u_i : i \in \mathbb{N} \}$ consists of independent random variables, each of which has, for small enough $c$, probability at least $c$ of exceeding $n^{1/2}$. 
To see this, note that the value $u_{j+1}$ will be recorded on the first return to $\{ x = 1/2 \}$  after that at which $u_{j}$ is recorded, provided that this return occurs at a higher value of $y$, which occurs with probability at least $2^{-1}\big( 1 + o(1) \big)$, in which case, the difference $u_{j+1} - u_j$ will exceed $n^{1/2}$ with positive probability, by the central limit theorem. 
If  a proposed entry $u_j$ exceeds $m$, then
the wraparound time has already occurred, and the entry is not recorded.
We see that 
(\ref{riest}) follows. 

Suppose now that $\rho > 0$. 
Then for any $k$,
\begin{equation}\label{qest}
\mathbb{P}\big( J_2 > k \big) \leq \gamma^k,
\end{equation}
for some $\gamma = \gamma(\rho) \in (0,1)$. Indeed, it is readily seen that a new term is added to the $v$-sequence, independently of its history, with a probability that is bounded away from one. So $J_2$ satisfies (\ref{riest}) in this case also.  

This completes the proof of the last assertion of the proposition. 
\end{proof}

The number of cycles experiences a rapid decline as the value of 
$m$ is increased beyond that treated in Proposition \ref{propone}.

\begin{prop}\label{proptwo}
We set $m = n + \lfloor C \sqrt{n\log n} \rfloor$, for $C \in (0,\infty)$ a fixed constant. Then
$$
 \E\big(N_{(1,1)}\big) = \frac1{2\sqrt{\pi}}n^{\frac{2-C^2}{4}}\bigg( 1+ 
   O \Big( \frac{(\log n)^3}{n^{1/2}} \Big) \bigg).
$$
\end{prop}
\noindent{\bf Remark.}
The principal interest of this result is for values $C \in (0,\sqrt{2})$, in which case, all cycles are of homology class $(1,1)$
with high probabilty.
Indeed, the proof of the first part of Proposition \ref{propone}
may be adapted to show that every cycle in the configuration is a $(1,1)$-cycle with probability exceeding
$$
1 - \exp \bigg\{ - c \frac{n^{1/2 - C^2/4}}{\sqrt{\log n}} \bigg\},
$$
where $c > 0$ is a small constant. \\
\begin{proof} 
We estimate the expected number $\mathbb{E} N_{(1,1)}$ of $(1,1)$ cycles as follows. Every $(1,1)$ cycle contains exactly one of the edges $(0,y)(1,y),$ for $y=0,\dots, m-1$.
Hence, the number of $(1,1)$ cycles is equal to the number of edges $(0,y)(1,y)$ 
with $0 \leq y < m$ that are present in the CRSF configuration and are such that 
the trajectory begun at $(1,y)$ first visits the line $\{ x = 1/2 \}$ at $(1/2,y)$. 
Thus,
\begin{displaymath}
 \mathbb{E} \big( N_{(1,1)} \big) = m \mathbb{P} \Big( \big\{ Z = m \big\} \cap \big\{ (0,0)(1,0) \in \mathcal{C} \big\} \Big),
\end{displaymath}
where $\mathcal{C}$ denotes the CRSF configuration, and  $Z$ denotes the $y$-coordinate of the trajectory starting at 
$(1,0)$ on its first return to the line $x=\frac12$. Noting that $Z = m$ implies that $(0,0)(1,0) \in \mathcal{C}$, we obtain 
\begin{equation}\label{mpform}
\mathbb{E} \big( N_{(1,1)} \big) = m \mathbb{P} \big( Z = m \big).
\end{equation}
Note further that $Z = \sum_{i=1}^n Y_i$, 
where $\big\{ Y_i: i \in \{ 1,\ldots, n \}  \big\}$
is an independent sequence of geometric random variables of mean one and variance two. 
We find that
$$
\mathbb{P} \big( Z = m  \big)
   = \mathbb{P} \Big(   \sum_{i=1}^n{\big( Y_i - 1 \big)}
  = \lfloor C \sqrt{n \log n} \rfloor   \Big).
$$
We require a local limit theorem for a sum of independent identicaly distributed random variables in a regime of moderate deviations. Theorem 3 of \cite{Richter} yields
\begin{displaymath}
 \mathbb{P} \Big(  \sum_{i=1}^n{\big( Y_i - 1 \big)} \Big) = 
\frac{1}{2 \sqrt{\pi n}} n^{-\frac{1}{2} - \frac{C^2}{4}}
 \bigg( 1 +  O \Big( \frac{(\log n)^3}{n^{1/2}} \Big)  \bigg),
\end{displaymath}
from which the result follows. 
\end{proof}

\section{The secondary spike when $m\approx n$}\label{sqrtlog}

We prove two propositions regarding the behavior 
of the model in a regime where $m= n + C \sqrt{n\log n}$, 
with $\vert C \vert \in (\sqrt{2},\infty)$ a fixed constant. 
In Proposition \ref{propthree}, we show that it is likely that 
there is a cycle of length $n^{3/2 + o(1)}$, and, in
Proposition \ref{propfour}, we establish that it is likely to be the only cycle.

\begin{prop}\label{propthree}
For $\vert C \vert > \sqrt{2}$, set $m = n + C \sqrt{n\log n}(1+o(1))$. 
Then, for any $\epsilon > 0$ and for $n$ sufficiently large,
\begin{displaymath}
 \mathbb{P} \Big( \textrm{any cycle has length 
 at least $\frac{n^{3/2}}{3|C| \sqrt{\log n}}$} \Big)
\geq 1 - n^{\frac12-\frac{C^2}{4} + \epsilon}.
\end{displaymath}
\end{prop}

\begin{proof}
We treat the case that $C < 0$, the other being similar.
Let $\phi^*: \{0,\ldots,m-1\} \to \mathbb{N}$ be the $y$-coordinate of the 
return map 
of the line $x=0$ to itself.
That is, let $\big( 0, \phi^*(y) \big)=\phi^\tau(0,y)$ where $\tau=\tau_y>0$ is the first time that
$\phi^\tau (0,y )$ has $x$-coordinate zero after the first positive time at which it has a strictly positive $x$-coordinate. Note that $y<y'$ implies $\phi^*(y)\le
\phi^*(y')$, in other words $\phi^*$ is non-decreasing.

For $i \in \mathbb{N}$, let $D_i$ denote the event that 
$$
  \phi^*\big( \lfloor i \epsilon C \sqrt{n \log n} \rfloor \big) > (i+1)\epsilon C \sqrt{n \log n}. 
$$
Note that, for any given $y$, we may write $\phi^*(y) = y -m+ \sum_{i=1}^n Y_i $, where $\big\{ Y_i: i \in \mathbb{N} \big\}$
is an independent sequence of geometric random variables of mean one and variance two. By Theorem 5.23 of \cite{Petrov}, we have the bound
\begin{equation}\label{diineq}
 \mathbb{P} \Big( D_i \Big) \geq 1 - C_0 n^{- \frac{C^2}{4}(1-\epsilon)^2},
\end{equation} 
for some large constant $C_0$.
Set $D = \bigcap_{i=0}^{\lfloor \frac{m}{C\epsilon\sqrt{n \log n}} \rfloor} D_i$.
We claim that, if $D$ occurs, then $\phi^*(k) \geq k$ 
for all $k \in \{ 0,\ldots,m-1 \}$. Suppose on the contrary $\phi^*(k)<k$. 
Let $j \in \mathbb{N}$ be maximal such that $\lfloor j \epsilon C \sqrt{n \log n} \rfloor \leq k$. 
Then 
\begin{displaymath}
 \phi^* \big( \lfloor j \epsilon C \sqrt{n \log n} \rfloor \big)
 \leq \phi^*(k) < k<\lfloor (j+1)\epsilon C \sqrt{n \log n} \rfloor< \phi^* \big( \lfloor j \epsilon C \sqrt{n \log n}  \big),
\end{displaymath} 
a contradiction.

For $i \in \mathbb{N}$, let $E_i$ denote the event that
\begin{displaymath}
 \phi^* \big( \lfloor i C \sqrt{n \log n} \rfloor \big)
 \leq \lfloor  (i+2) C \sqrt{n \log n} \rfloor.
\end{displaymath}
Arguing similarly to (\ref{diineq}), we note that
\begin{equation}\label{eiineq}
\mathbb{P}\Big( E_i \Big) \geq 1 - C_0 n^{-\frac{C^2}{4}},
\end{equation}
where $C_0$ again denotes a large constant. 
 
Set $E = \bigcap_{i=0}^{\lfloor \frac{m}{C\sqrt{n \log n}} \rfloor} E_i$.

Define $y_0=0$ and for $i>0$ define $y_i=\phi^*(y_{i-1})$.
We will show that, if $E$ occurs,
\begin{equation}\label{hes}
 y_{i+1} - y_i \leq 3 C \sqrt{n \log n}
\end{equation}
for each $i \in \mathbb{N}$. Let $j \in \big\{ 1,\ldots, \lfloor \frac{m}{C \sqrt{n \log n}} \rfloor + 1\big\}$ satisfy
\begin{displaymath}
  \lfloor  (j-1) C \sqrt{n \log n} \rfloor
 \leq y_i < \lfloor  j C \sqrt{n \log n} \rfloor.
\end{displaymath} 
We have that 
\begin{eqnarray}
  y_{i+1} & = & \phi^* \big( y_i \big) \leq \phi^* 
  \big( \lfloor j C \sqrt{n \log n} \rfloor \big) \nonumber \\
  & \leq & \lfloor (j+2)C \sqrt{n \log n} \rfloor \leq y_i + 3C\sqrt{n \log n}, \nonumber 
\end{eqnarray}
the first inequality since $\phi^*$ is non-decreasing, the second due to the occurrence of $E_j$. We have obtained (\ref{hes}). 

Let $K \in \mathbb{N}$ be maximal subject to $y_K < m$ (note that $K$ is finite if $D$ occurs). We claim that, on the event $D$, any cycle has at least $K$ strands, and that, on the event $E$, $K \geq \frac{m}{3C \sqrt{n \log n}}$.

Indeed, setting $C_i = \{ y_i , \ldots, y_{i+1}  \}$ for $i \in \{ 0,\ldots,K - 1 \}$, we have that $\phi^*\big( C_i \big) \subseteq C_{i+1}$ for such $i$, if $D$ occurs. The monotonicity of $\phi^*$ implies that any cycle contains a point $(0,a_1)$ with $a_1 \in C_1$ and, by the sequence of inclusions, distinct points $(0,a_i)$ with $a_i \in C_i$ for each value of $i$. Hence, the cycle has at least $K$ strands.

The lower bound on $K$ follows by noting that, from (\ref{hes}), 
\begin{displaymath}
  m - 3C \sqrt{n \log n} \leq y_K 
 = \sum_{i=1}^{K-1} \big( y_{i+1} - y_i \big) \leq 3C (K-1) \sqrt{n \log n}.
\end{displaymath} 

The proof is completed by noting the following bounds on $\mathbb{P}(D)$ and $\mathbb{P}(E)$, which follow from (\ref{diineq}) and (\ref{eiineq}):
\begin{displaymath}
 \mathbb{P}(D) 
 \geq 1 - \frac{2 C_0 n^{1/2 - \frac{C^2}{4}(1 -\epsilon)^2}}{C \epsilon \sqrt{\log n}}
\end{displaymath}
and
\begin{displaymath}
 \mathbb{P}(E) 
\geq 1 - \frac{2 C_0 n^{1/2 - \frac{C^2}{4}}}{C \sqrt{\log n}}.\nonumber
\end{displaymath}
\end{proof}

\begin{prop}\label{propfour}
Set $m = n + C \sqrt{n\log n}\big(1+o(1)\big)$ with $\vert C \vert > \sqrt{2}$. Then, in the CRSF,
for each $\epsilon > 0$,
\begin{displaymath}
 \mathbb{P} \Big( \textrm{there exist at least two disjoint cycles} \Big)
 \leq n^{\frac{1}{2} - \frac{C^2}{4} + \epsilon}
\end{displaymath}
for $n$ sufficiently large.
\end{prop}
\noindent{\bf Remark.} As the proof will show, in the presence of one cycle of the length given in the statement of Proposition \ref{propthree}, the conditional probability of another cycle decays as $\exp \big\{ - n^{1/2 + o(1)} \big\}$. 
In the regime that Propositions \ref{propthree} and \ref{propfour} treat, then, the most probable means by which two cycles form is by local fluctuations in the generating random walks that create two $(1,1)$-cycles. This occurs with a probability that decays polynomially in $n$.\\ 
\proof
Recall the events $D$ and $E$, the quantity $K$ and the intervals $C_i$ from the proof of Proposition \ref{propthree}. Suppose that the event $D \cap E$ occurs.
Let $a_i$ be the intersection of a cycle with the interval $C_i$.
Then 
$$a_{i+1}-a_i\leq y_{i+2}-y_i \leq 6C\sqrt{n\log n}.$$

Suppose that the CRSF configuration is formed by firstly running the random walk $W_{0,0}$ until it meets its own trace, and then running the walks $W_{0,w_i}$ until existing trees or the current trace is hit, where $w_i \in \big\{ 0,\ldots, m - 1 \big\}$ are selected in an arbitrary manner from the subset of the line $\big\{ x = 0 \big\}$ not yet belonging to any tree.

If two cycles are to be present in the configuration, then, for some 
$z_1$ with $a_1 < z_1 < a_2$, $W_{0,z_1}$ must not meet the first cycle before it hits its own trace. Set $z_0 = w_1$, and let $z_i$ denote the $y$-coordinate of the $i$-th return of the walk $W_{0,w_1}$ to the line $\big\{ x = 0 \big\}$.

In the event $D \cap E$, if $W_{0,w_1}$ does not meet the first cycle before visiting its own trace, then 
\begin{equation}\label{azaeqn}
 a_i < z_i < a_{i+1}
\end{equation}
for each $i \in \big\{ 1,\ldots, K - 2 \big\}$.

We sample the sequence $z_i$ when $i$ is a multiple of $L=\lfloor
n^\eps\rfloor$.
Note that (\ref{azaeqn}) implies that
\begin{equation}\label{zmza}
|a_{(j+1)L}-a_{jL + 1}|\le |z_{(j+1)L}-z_{jL}|\le |a_{(j+1)L + 1}-a_{jL}|
\end{equation}
for each $j$. So $z_{(j+1)L}-z_{jL}$ is restricted to 
an interval of length at most $12C\sqrt{n\log n}$.

The quantity $z_{(j+1)L} - z_{jL}$ has the distribution 
of $\sum_{i=1}^{nL}X_i -mL$, where $\big\{ X_i \big\}$ is a sequence of independent geometric random variables of mean one and variance two (these are the vertical displacements of the walk $W_{0z}$ in between successive rightward movements). 

Now (\ref{zmza}) at the given value of $j$ requires that this sum 
$\sum_{i=1}^{nL} X_i$
of independent random variables lie in a fixed interval of length at most $12 C \sqrt{n \log n}$. It follows readily from Theorem 3 of \cite{Richter}
that the probability of this event is maximized by choosing the interval to be centred at $nL$, and, thus, to be bounded above by 
$$
 C_0 \frac{12 C \sqrt{n \log n}}{\sqrt{nL}} \leq n^{-\frac{\epsilon}{2}}.
$$
Thus, on the event $D \cap E$, the probability that each of the inequalities (\ref{zmza}) is satisfied is at most 
$$
 \big( n^{- \epsilon/2} \big)^{n^{1-\eps}}
  \leq \exp\{ - n^{1- \epsilon}\}.
$$
The bounds on the probabilities of $D$ and $E$ presented in the proof of Proposition \ref{propthree} complete the proof. $\Box$

\section{Near $m/n=p/q$}\label{p/q}

We extend the previous results to the case $m/n$ is near a rational $p/q$ with
small denominator.

\begin{prop}\label{propfive}
Let $p,q$ be fixed and relatively prime.
Let $\rho\in(0,\infty)$ be fixed and $m = (p/q)n + \rho \sqrt{n}(1+o(1))$. 
For $c > 0$ small enough, each closed orbit has homology class $(p,q)$ with
probability at least $1 - \exp \big\{ - cn^{1/2} \big\}$, while
$$
\Prob\Big(N_{(p,q)} > cn^{1/2}\Big) \geq 1 - \exp \Big\{ - c n^{1/2} \Big\}
$$
for sufficiently large $n$.
For $C > 0$ sufficiently large, we have
$$
\Prob\Big(N_{(p,q)} > C n^{1/2}\Big)  \leq \exp \Big\{ - C^{-1} n^{1/2} \Big\}   
$$
for sufficiently large $n$.
\end{prop}

\begin{proof} 
The first part of the proof is essentially the same as the proof of Proposition \ref{propone},
with the following changes. We again partition
the torus into strips, but in this case the strips have horizontal length $pn$
instead of $n$. Thus each strip winds $p$ times around horizontally, and 
$q$ times vertically, before 
closing up. The direction of the strip is now parallel to the closed curve
of homology class $(p,q)$ on the torus, and the width of the strips 
is still $2K\sqrt{n}$ for some large $K$.

For the second half of the proof, 
we require some variations on the sequence of maxima and minima that we record.
We divide returns to the line $\big\{ x = 1/2 \big\}$ into $p$ classes, according to the value of the index of the return reduced mod $p$. We then form $p$ separate lists $\{ u_i^j \}, \{ v_i^j  \}$ of maxima and minima, where the $y$-coordinate of the $k$-th return to the line $\big\{ x = 1/2 \big\}$ is entered as a maximum $u^j_i$ or as a minimum $v_i^j$ on the list $j$, $j = k \, \textrm{mod} \, p$, if this $y$ value exceeds, or is less than, any $y$-coordinate for an $l$-return to the line $\big\{ x = 1/2 \big\}$ with $l \, \textrm{mod} \, p$ equal to $j$.

We define the wraparound time to be that moment at which there no longer exists a cycle of homology $(p,q)$ that is disjoint from the existing trace of $W_{00}$. We no longer record the $y$-coordinate of a return to the line $\big\{ x = 1/2 \big\}$ after the wraparound time.

Let $\big\{ u_i^j: 0 \leq i \leq J_1^j \big\}$ and 
$\big\{ v_i^j:  1 \leq  i \leq J_2^j  \big\}$ 
denote the maxima and minima recorded on the $j$-th list. Similarly to the case treated in Proposition \ref{propone},
the $y$-coordinate of each horizontal edge crossing $\{ x = 1/2 \}$ is recorded on one of the lists, for each cycle in the configuration, except possibly the last one.  We learn that
$$
N \leq \sum_{j=1}^p \big( J_1^j + J_2^j + 1 \big) \, + p.
$$
The proof is completed by estimating the tail of the random variables $J_1^j$ and $J_2^j$ as in the previous proof. 
\end{proof}

\begin{prop}\label{pq2}
We set $m = (p/q)n + C \sqrt{n} \sqrt{\log n}$, for $C \in \big( 0,\sqrt{2p} \big)$. Then$$
\mathbb{E} N_{(p,q)}  = \frac{\sqrt{p}}{2q\sqrt{\pi}} n^{\frac{1}{2} - \frac{C^2}{4p}}  \bigg( 1 +  O \Big( \frac{(\log n)^3}{\sqrt{n}} \Big) \bigg).
$$
\end{prop}
\begin{proof}
We replace the $(1,1)$-cycles considered in the proof of Proposition \ref{proptwo}  by $(p,q)$-cycles, and note the following variation: 
we have
$$\E\big( N_{(p,q)} \big)= m \mathbb{P}\Big( \big\{ Z= pm \big\} \cap \big\{ (0,0)(1,0) \in C \big\} \cap A \Big),$$
where, in this instance, $Z$ denotes the $y$-coordinate of the trajectory starting at $(1,0)$ on its $p$-th return to the line $\big\{ x = 1/2 \big\}$. The event $A$ is that the walk $W_{0,0}$, after visiting $(0,1)$, does not meet itself before its $p$-th return to the line $\big\{ x = 1/2 \big\}$. Noting that 
$\mathbb{P}\big( A \big\vert (0,0)(0,1) \in C \big) \geq 1 - \exp \big\{ - cn \big\}$, and that $Z = pm$ implies that $(0,0)(0,1) \in C$, we see that 
$$
\mathbb{E}\big(  N_{(p,q)}  \big) = m \mathbb{P} \big( Z = pm \big) \bigg( 1 +  O \Big( \exp \big\{ - cn \big\} \Big) \bigg).
$$ 
Noting that $Z = pm$ if and only if $\sum_{i=1}^{pn} X_i = pn + p\lfloor C \sqrt{n \log n} \rfloor$, and applying Theorem 3 of \cite{Richter}, 
we find that
$$
\mathbb{P} \Big( Z = pm \Big) = \frac{1}{2\sqrt{\pi p}} n^{- \frac{1}{2} - \frac{C^2}{4p}} \bigg( 1 +  O \Big( \frac{(\log n)^3}{\sqrt{n}} \Big) \bigg).
$$
We thus have
$$
\mathbb{E} \big( N_{(p,q)} \big)  = \frac{\sqrt{p}}{2q\sqrt{\pi}} n^{\frac{1}{2} - \frac{C^2}{4p}}  \bigg( 1 +  O \Big( \frac{(\log n)^3}{\sqrt{n}} \Big) \bigg).
$$
\end{proof}

The next two propositions, whose proofs mimic those of Propositions \ref{propthree} and \ref{propfour}, treat the secondary spike for a torus with aspect ratio close to a general rational.
\begin{prop}\label{prop}
Let $p,q \in \mathbb{N}$ satisfy $(p,q) = 1$. Set
$m = (p/q)n + C \sqrt{n\log n},$
for $\vert C \vert>\sqrt{2p}.$
Then, for $K = K(p,q),$
$$
 \mathbb{P} \Big(  \textrm{there exists a loop of length at least $\frac{n^{3/2}}{\sqrt{\log n}}$} \Big) \geq 1 - n^{\frac{1}{2} - \frac{C^2}{4p} + \epsilon},
$$
as $n \to \infty$.
\end{prop}
\begin{prop}\label{propeight}
Set $m = (p/q)n + C \sqrt{n\log n}\big(1+o(1)\big)$ with $\vert C \vert > \sqrt{2p}$. Then, in the CRSF,
for each $\epsilon > 0$,
\begin{displaymath}
 \mathbb{P} \Big( \textrm{there exist at least two disjoint cycles} \Big)
 \leq n^{\frac{1}{2} - \frac{C^2}{4p} + \epsilon},
\end{displaymath}
for $n$ sufficiently large.
\end{prop}

\section{The irrational regime}\label{irrationalsection}

Let $m,n \in \mathbb{N}$ with $Cn > m > n$ for a constant $C>1$.

We begin by collecting some elementary facts about continued fractions.
These can be found in, for example, \cite{HW}.
Let 
$$\frac{m}{n}=a_0+\frac1{a_1+\frac1{a_2+\dots+\frac1{a_l}}}$$
be the continued fraction decomposition of $m/n$. 
Define $p_0/q_0=a_0$ and for $0\leq j\leq l$, define
$$\frac{p_j}{q_j}=a_0+\frac1{a_1+\frac1{a_2+\dots+\frac1{a_j}}}$$
to be the rational approximants to $m/n$.

We have $\frac{p_k}{q_k}<\frac{m}n\leq \frac{p_{k+1}}{q_{k+1}}$
for $k \leq l$ even and $\frac{p_{k+1}}{q_{k+1}}\leq\frac{m}n< \frac{p_k}{q_k}$
for $k \leq l$ odd. In each case, $\frac{p_{k+1}}{q_{k+1}}$ is the closer endpoint to $m/n$.
Also $|\frac{p_k}{q_k}-\frac{p_{k+1}}{q_{k+1}}|=\frac1{q_kq_{k+1}}.$
Hence, 
\begin{equation}\label{approx}
\frac{n}{2q_{j+1}}<|np_{j}-mq_{j}|<\frac{n}{q_{j+1}}
\end{equation}

We also have
\begin{eqnarray*}
p_k&=&a_kp_{k-1}+p_{k-2}\\
q_k&=&a_kq_{k-1}+q_{k-2}.
\end{eqnarray*}

Choose $j_0$ so that 
\begin{equation}\label{j0def}
q_{j_0}\leq n^{1/3}<q_{j_0+1}.
\end{equation}

For typical $m,n$ the $a_k$ are $O(1)$; in fact the Gauss-Kuz'min law says
that the probability that $a_k>x$ is of
order $1/x$ for $x$ large. Moreover $j_0$ is typically of order $C\log n$
for a (known) constant $C$, see \cite{Bill}.

\begin{theorem} Define $j_0$ as in (\ref{j0def}).
For each $k \in \mathbb{N}$, there exists $c = c(k) > 0$ independent of $n$ and $m$, such that the probability that there at least $k$ cycles, each  of homology class $(p_{j_0},q_{j_0})$, is at least $c$.
\end{theorem}
\noindent{\bf Remark.} The length of each such cycle is equal to $np_{j_0} + mq_{j_0} \geq  n q_{j_0} \geq  n q_{j_0 + 1}/(a_{j_0} + 1) > n^{4/3}/(a_{j_0} + 1)$. Since $a_{j_0}$ is typically $O(1)$, we see that, for $n$ large and most choices of $m$, cycles of length $n^{4/3}$ form with positive probability in the CRSF configuration. 

\proof
We argue the case $k = 2$, the general one being no harder. 
Note that
$$
\Big\vert \frac{p_{j_0}}{q_{j_0}} - \frac{m}{n} \Big\vert
 < \frac{1}{q_{j_0} q_{j_0+1}},
$$
so that 
\begin{equation}\label{pqmn}
\Big\vert \frac{p_{j_0}}{q_{j_0}} - \frac{m}{n} \Big\vert
 < \frac{1}{q_{j_0} n^{1/3}},
\end{equation}
which implies, by $1 \leq \frac{m}{n} \leq C$ and (\ref{j0def}), that
$$
 \frac{q_{j_0}}{2} \leq p_{j_0} < (C+1) q_{j_0}.
$$ 
Using (\ref{j0def}), then,
\begin{equation}\label{pj0}
\frac{p_{j_0}}{C + 1} \leq n^{1/3}. 
\end{equation}
Note that, by (\ref{pqmn}),
$|np_{j_0}-mq_{j_0}|<n^{2/3}$.
Let $R_1$ be the closed line/loop through the origin
of slope $mq_{j_0}/np_{j_0}$ on the torus.
The vertical distance between strands of $R_1$ is $m/p_{j_0}\ge \frac{n^{2/3}}{C+1}$. 
Let $S_1$ be a strip centered on $R_1$ and of width $\frac{n^{2/3}}{3(C+1)}$.
Let $S_2$ be a translate of $S_1$ which is disjoint from $S_1$.

The probability that $W_{0,0}$ does not exit $S_1$ before making a cycle
can be bounded as follows. 
It suffices that: \begin{enumerate}
\item The walk stays in the strip for two circuits, i.e. 
for $2p_{j_0}$ returns to the line $x=0$.
\item At the end of the first circuit (at the $p_{j_0}$-th return) 
the walk is in the upper half of the strip.
\item At the end of the second circuit the walk is in the lower half of the strip.
\end{enumerate}
 In the notation used to argue that $\mathbb{P} \big( E_i \big) \geq c$
in the proof of Proposition \ref{propone},
the first event is
$$\left\{\sup_{\ell\in\{1,2,\dots,2p_{j_0}n\}}\left|
\sum_{i=1}^{\ell}X_i-\ell \frac{mq_{j_0}}{np_{j_0}} \right|\leq \frac{n^{2/3}}{6(C+1)} \right\}.$$
Using $|1-\frac{mq_{j_0}}{np_{j_0}}|\leq \frac1{p_{j_0}n^{1/3}}$, 
we may now argue similarly to the deduction of $\mathbb{P}(E_i) \geq c$ 
that the three events listed above occur simultaneously with a probability that is positive, uniformly in $n \geq n_0$ and $n < m < Cn$. $\Box$

\begin{theorem} Fix $\eps>0$, and 
choose $j$ so that $q_j \leq n^{1/3-\eps/2}< q_{j+1}$.
Suppose that $a_{j}<n^{\eps/2}$ and  $a_{j+1} \leq n^{\epsilon/4}$. 
Then
$$\mathbb{P}\Big(\text{there is a cycle of length
at most $O\big(n^{4/3-\eps}\big)$}\Big) \leq \exp \big\{ -c n^{\eps/2} \big\}.$$
\end{theorem}

\begin{proof}
From (\ref{approx}) we have $|np_j-mq_j|>\frac{n}{2q_{j+1}}$. This implies that
after $p_j$ traversals, the strands of the ray $R$ of slope $1$ starting at the origin
do not come within $\frac{n}{2q_{j+1}}$ of each other. 
In particular, 
there is an embedded strip $U=U_j$, centered on the ray $R$ starting at the origin, 
of width $\frac{n}{2q_{j+1}}$ and 
horizontal length $np_j$. 

By assumption, $a_j<n^{\eps/2}$
so that $q_{j+1}<(a_j+1)q_j< \big( n^{\eps/2}+1 \big) n^{1/3-\eps/2}< 2 n^{1/3}$.
This implies that the strip $U_j$ has width at least $n^{2/3}/4$.
Arguing similarly to (\ref{pj0}),
\begin{equation}\label{pjl}
 \frac{p_j}{C  + 1} \leq n^{1/3 - \epsilon/2}.
\end{equation}
Suppose that the CRSF configuration is formed by firstly running the walk $W_{0,0}$ until it hits its trace. If the walk $W_{0,0}$ remains in the strip $U_j$ until it hits itself, forming a cycle, then the cycle to which $(0,0)$ belongs has length at least $m q_j \geq n q_{j+1}/(a_{j+1} + 1) \geq 2^{-1} n^{4/3 - (3/4) \epsilon}$, since $a_{j+1} \leq n^{\epsilon/4}$. 
Each cycle has the same length. Hence,  the event that there is a cycle of length at most $n^{4/3- \eps}$ implies that
the walk $W_{0,0}$ leaves the strip $U_j$ before reaching its end. The strip $U_j$ having horizontal length $np_j \leq (C+ 1)n^{4/3 - \epsilon/2}$ by (\ref{pjl}), we see that, with $\big\{ X_i: i \in \mathbb{N}\big\}$ being defined in the proof of Proposition \ref{propone}, if 
$\big\vert \sum_{i=1}^j X_i - j\big\vert \leq n^{2/3}/8$ for 
each $j \in \big\{ 1,\ldots, \lfloor (C+1) n^{4/3 - \epsilon/2} \rfloor \big\}$, then the walk $W_{0,0}$ remains in $U_j$ until reaching its end.  

A brief argument using Theorem 5.23 of \cite{Petrov} yields
$$
\mathbb{P}\left(\max_{0\leq j\leq (C+1) n^{4/3-\eps/2}} \left|\sum_{i=1}^j X_i -j\right|
< n^{2/3}/8
\right)
>1-e^{-c n^{\eps/2}}$$ which completes the proof.
\end{proof}

\begin{theorem}
Fix $\eps>0$ and
choose $k<\ell$ so that $p_{k}\le n^{1/3-\frac{7}{48}\eps}<p_{k+1}$
and $p_{\ell}\le n^{1/3+\frac{47}{48}\eps}<p_{\ell+1}$. 
Suppose that $\max \big\{ a_{k+1},a_{l+1} \big\} <n^{\frac{5}{48}\eps}$.  Then 
$$\mathbb{P}\Big(\text{there is a cycle of length
at least $n^{4/3+ \eps})$}\Big) \leq  \exp \Big\{ -n^{\frac{\eps}{7}} \Big\}.$$
\end{theorem}

\begin{proof}
We in fact prove a modified statement: choose $k \leq \ell$ so that $p_{k}\le n^{1/3-\beta\eps}<p_{k+1}$
and $p_{\ell}\le n^{1/3+\gamma\eps}<p_{\ell+1}$, and  
suppose that $\max \big\{ a_{k+1},a_{l+1} \big\} <n^{\alpha\eps}$. Let $\delta > \alpha + \beta$ with $\beta > \alpha$. Then 
\begin{equation}\label{pefeqn}
\mathbb{P}\Big(\text{there is a cycle of length
at least $3 n^{4/3+\gamma\eps}$}\Big)\leq  e^{-n^{\beta\eps+o(1)}}+e^{-n^{(\gamma-\alpha-2\delta)\eps+o(1)}}.
\end{equation}
The statement of the proposition then follows by the choices $\alpha = 5/48$, $\beta = 7/48$, $\delta = 13/48$ and $\gamma =47/48$.

Let $E$ be the event 
$$E=\left\{\max_{j\in\{1,\dots,np_{k}\}}\left|\sum_{i=1}^j X_i-j\right|\leq 
n^{2/3}\right\}$$
and $F$ be the event
$$F=\left\{\max_{j\in\{1,\dots,np_{\ell}\}}\left|\sum_{i=1}^j X_i-j\right|\geq 
n^{2/3+\delta\eps}\right\}.$$

Suppose that the CRSF configuration is formed by firstly running the walk $W_{0,0}$ until it hits its trace.
We claim that, on the event $E\cap F$, the walk $W_{0,0}$ completes a cycle before its $q_l$-th return to the line $\big\{ x = 0 \big\}$. 
To see this, note that after $p_{k}$ returns to $x=0$, the ray $R$ splits
the $x$-axis 
into intervals of lengths between $m/p_{k+1}$ and $m/p_k$.
We have
$$\frac{n^{1-\alpha\eps}}{2p_k}
<
\frac{m}{p_{k+1}}
< 
\frac{m}{n^{1/3}} n^{\beta \epsilon}
<
\frac{m}{p_k}
<
\frac{2 C n^{1+\alpha\eps}}{p_{k+1}},$$
so these intervals are
of order at least $m/p_{k+1} \geq n/p_{k+1} \geq \frac{m}{2C n^{1/3}} n^{(\beta - \alpha) \epsilon} \geq (2C)^{-1} n^{2/3 + (\beta - \alpha) \epsilon}$ and at most $m/p_k \leq Cn/p_k \leq 2C \frac{m}{n^{1/3}} n^{\beta \epsilon} n^{\alpha \epsilon} \leq  2 C^2 n^{2/3+(\alpha+\beta)\eps}$.
Due to the occurrence of  $E$, the path $W_{0,0}$ lies within $n^{2/3}$
of these points, and, due to $\beta > \alpha$, the path of $W_{0,0}$ does not intersect itself before time $n p_k$.
For the event $F$, at some point before the $p_{\ell}$-th return
there is a displacement of at least $n^{2/3+\delta\eps}$
from $R$. Since $\delta>\alpha+\beta$ the path must intersect itself. 

The number of horizontal steps in the cycle to which $(0,0)$ is rooted is at most $n p_l \leq  n^{4/3 + \gamma \epsilon}$. We have demonstrated that $W_{0,0}$ hits its trace before the first moment $j$ at which $\big\vert \sum_{i=1}^j X_i - j \big\vert \geq n^{2/3 + \delta \epsilon}$. We learn that the number of vertical steps in the cycle is at most $n p_l + n^{2/3 + \delta \epsilon} \leq 2  n^{4/3 + \gamma \epsilon}$. Hence, on the event $E \cap F$, the cycle to which $(0,0)$ is rooted has length at most $3  n^{4/3 + \gamma \epsilon}$. 

By invariance under vertical translation, the probability that there exists a cycle whose length exceeds $3  n^{4/3 + \gamma \epsilon}$ is at most $m \mathbb{P} \big( E \cap F \big)$.

With the aid of Theorem 5.23 of \cite{Petrov},
$$
\mathbb{P} \big( E  \big) \geq 1-e^{-n^{\beta\eps+o(1)}}.
$$
By  $a_{l+1} < n^{\alpha \epsilon}$ and the central limit theorem,
$$
\mathbb{P} \big( F \big) \geq 1-e^{-n^{(\gamma-\alpha-2\delta)\eps+o(1)}}.
$$
Hence, we obtain (\ref{pefeqn}).
\end{proof}

\end{document}